\documentclass[12pt,reqno]{amsart}
\usepackage{fullpage}
\usepackage{color}
\usepackage[colorlinks=true, pdfstartview=FitH, linkcolor=blue, citecolor=blue, urlcolor=blue]{hyperref}

\newtheorem{theorem}{Theorem}[section]
\newtheorem{lemma}[theorem]{Lemma}
\newtheorem{prop}[theorem]{Proposition}
\newtheorem{cor}[theorem]{Corollary}
\theoremstyle{definition}
\newtheorem{definition}[theorem]{Definition}
\newtheorem{example}[theorem]{Example}
\newtheorem*{remark}{Remark}

\renewcommand{\mod}[1]{{\ifmmode\text{\rm\ (mod~$#1$)}\else\discretionary{}{}{\hbox{ }}\rm(mod~$#1$)\fi}}
\newcommand{\ep}{\varepsilon}
\newsymbol\nmid 232D

\newcommand{\A}{{\mathcal A}}

\newcommand{\D}{{\mathcal D}}

\newcommand{\N}{{\mathbb N}}
\renewcommand{\P}{{\mathcal P}}

\newcommand{\R}{{\mathbb R}}

\newcommand{\mfM}{{\mathfrak M}}
\newcommand{\mfK}{{\mathfrak K}}

\newcommand{\li}{\operatorname{li}}

\begin{document}

\title{The distribution of sums and products of \\ additive functions}
\author[Greg Martin and Lee Troupe]{Greg Martin and Lee Troupe}

\begin{abstract}
The celebrated Erd\H os--Kac theorem says, roughly speaking, that the values of additive functions satisfying certain mild hypotheses are normally distributed. In the intervening years, similar normal  distribution laws have been shown to hold for certain non-additive functions and for amenable arithmetic functions over certain subsets of the natural numbers. Continuing in this vein, we show that if $g_1(n), \ldots, g_k(n)$ is a collection of functions satisfying certain mild hypotheses for which an Erd\H os--Kac-type normal distribution law holds, and if $Q(x_1, \ldots, x_k)$ is a polynomial with nonnegative real coefficients, then $Q(g_1(n), \ldots, g_k(n))$ also obeys a normal distribution law. We also show that a similar result can be obtained if the set of inputs $n$ is restricted to certain subsets of the natural numbers, such as shifted primes. Our proof uses the method of moments. We conclude by providing examples of our theorem in action.
\end{abstract}

\maketitle

\section{Introduction}

In 1939, Erd{\H o}s and Kac~\cite{ek40} used probabilistic methods to prove a striking fact concerning a class of strongly additive functions, establishing the following foundational result in the field of probabilistic number theory.

\begin{theorem}[Erd\H os/Kac]\label{EK classical}
Let $g$ be a strongly additive function with $|g(p)|\le1$ for all primes~$p$. Define
\begin{equation} \label{Ag and Bg def}
A_g(x) = \sum_{p\le x} \frac{g(p)}p \quad\text{and}\quad B_g(x) = \bigg( \sum_{p\le x} \frac{g^2(p)}p \bigg)^{1/2},
\end{equation}
and assume that $B_g(x)$ is unbounded. Then for every real number $u$,
\begin{equation} \label{EK limit}
\lim_{x\to\infty} \frac1x \# \big\{ n\le x\colon g(n) < A_g(x) + u B_g(x) \big\} = \frac1{\sqrt{2\pi}} \int_{-\infty}^u e^{-t^2/2}\,dt.
\end{equation}
\end{theorem}

\noindent
Here, as usual, a function $g\colon \N\to\R$ is {\em additive} if $g(mn)=g(m)+g(n)$ whenever $(m,n)=1$; 
an additive function~$g$ is {\em strongly additive} if $g(p^\alpha)=g(p)$ for every prime $p$ and positive integer~$\alpha$.
In particular, a strongly additive function is completely determined by its values on prime inputs.
For example, $\omega(n)$, the number of distinct prime factors of $n$, is a strongly additive function.

Theorem~\ref{EK classical} tells us that the values of the normalized version
$\big( g(n)-A_g(n) \big) / B_g(n)$
of~$g$ are distributed, in the limit, exactly like random real numbers chosen from the standard normal distribution of mean~$0$ and variance~$1$.
We codify this type of distributional law with the following terminology:

\begin{definition}\label{ek law def}
Let $g\colon \N\to\R$ and $A_g,B_g\colon \R_{\ge0}\to\R$ be functions. We say that {\em $g$ satisfies an Erd\H os--Kac law with mean $A_g$ and variance $B_g$} if equation~\eqref{EK limit} holds for every real number~$u$.
\end{definition}

Different proofs and generalizations of Erd{\H o}s and Kac's theorem abound. In a 1955 paper, Halberstam \cite{hal55} calculated the $m$th moments of the quantity $\big( g(n)-A_g(x) \big) / B_g(x)$ for $m \ge 1$, deducing the Erd{\H o}s--Kac theorem from these calculations. In 2007, Granville and Soundararajan \cite{gs07} calculated these moments using a technique both simpler and more adaptable than Halberstam's methods; as a result, one can prove Erd{\H o}s--Kac laws for strongly additive functions $g$ restricted to certain subsets of the natural numbers. An example of a result in this direction (though predating the work of Granville and Soundararajan) is due to Alladi~\cite{all87}, who proved an analogue of the Erd{\H o}s--Kac theorem for the usual prime-factor-counting functions $\omega$ and $\Omega$ over friable integers. 

In this paper, we adapt the methods of~\cite{gs07} to establish an Erd{\H o}s--Kac law for arbitrary sums and products of strongly additive functions satisfying certain standard hypotheses:

\begin{theorem}\label{main theorem intro version}
Let $Q(T_1,\dots,T_\ell)$ be a polynomial with nonnegative real coefficents.
Let $g_1, g_2, \ldots, g_\ell$ be nonnegative strongly additive functions such that:
\begin{enumerate}
\item each $g_j(p)\ll_j1$ uniformly for all primes~$p$;
\item for each $1\le j\le \ell$, the series $\sum_p {g_j^2(p)}/p$ diverges.
\end{enumerate}
Define 
\begin{equation} \label{main theorem mu kappa def}
\mu(g_j) = \sum_{p\le x} \frac{g_j(p)}p \quad\text{and}\quad \kappa(g_i,g_j) = \sum_{p\le x} \frac{g_i(p)g_j(p)}p\bigg( 1-\frac1p \bigg) \quad\text{and}\quad \sigma^2(g_i) = \kappa(g_i,g_i),
\end{equation}
and suppose that $\sigma^2(g_i) \gg \mu(g_j)$ for sufficiently large $x$ for all $1 \leq i, j \leq \ell$.

Then $Q(g_1(n),\dots,g_\ell(n))$ satisfies an Erd\H os--Kac law with mean
\begin{equation} \label{AQx intro}
A_Q(x) = Q \big( \mu(g_1), \dots, \mu(g_\ell) \big) 
\end{equation}
and variance
\begin{equation} \label{BQx intro}
B_Q(x)^2 = \sum_{i=1}^\ell \sum_{j=1}^\ell \frac{\partial Q}{\partial T_i} \big( \mu(g_1),\dots,\mu(g_\ell) \big) \frac{\partial Q}{\partial T_j} \big( \mu(g_1),\dots,\mu(g_\ell) \big) \kappa(g_i,g_j) .
\end{equation}
\end{theorem}


While all of the hypotheses of the theorem are necessary for our proof in its full generality, the assumptions on the relative sizes of the $\mu(g_j)$ and $\sigma^2(g_i)$ could usually be relaxed in specific situations; moreover, the assumption that the $g_j$ be pointwise nonnegative is hardly used, other than to ensure that the covariances $\kappa(g_i,g_j)$ are nonnegative so that the right-hand side of equation~\eqref{BQx intro} has a square root in general. In the same vein, we work with strongly additive functions for simplicity, but in principle the methods could establish Theorem~\ref{main theorem intro version} for additive functions whose values on prime powers can depend upon the power as well as the prime.

In the last section of this paper, we give examples of some special cases of this theorem, such as the Erd\H os--Kac law for a product of additive functions (including the integer powers of an additive function); these results, which are justified in Section~\ref{examples sec}, continue to use the notation from equation~\eqref{main theorem mu kappa def}.

\begin{cor} \label{cor1}
Let $g(n)$ be a strongly additive function satisfying the hypotheses of Theorem~\ref{main theorem intro version}. For any positive integer $\delta$, the function $g(n)^\delta$ satisfies an Erd\H os--Kac law with mean $\mu^\delta(g)$ and variance $\delta^2 \mu(g)^{2\delta-2} \sigma^2(g)$.
\end{cor}

\begin{cor} \label{cor2}
Let $g_1, \ldots, g_\ell$ be strongly additive functions satisfying the hypotheses of Theorem~\ref{main theorem intro version}. Then for any $\ell \geq 1$, the function $g_1(n) \cdots g_\ell(n)$ satisfies an Erd\H os--Kac law with mean $\mu(g_1)\cdots \mu(g_\ell)$ and variance
\[
\big( \mu(g_1)\cdots \mu(g_\ell) \big)^2 \sum_{1\le i,j\le \ell} \frac{\kappa(g_i,g_j)}{\mu(g_i) \mu(g_j)}.
\]
\end{cor}

As it turns out, the proof of Theorem~\ref{main theorem intro version} is no harder if we replace $\{p\le x\}$ with any finite set $\P$ of primes. Similarly, we may replace the set of inputs $\{n\le x\}$ with a finite set $\A$, at the cost of introducing further multiplicative functions $h(n)$ that measure the local densities of~$\A$ but with no other significant changes. As remarked in~\cite{gs07}, this added flexibility allows for the derivation of several variants of the Erd{\H o}s--Kac theorem. As an example, we establish the following analogue of Theorem~\ref{main theorem intro version} for polynomials evaluated at additive functions on shifted primes:

\begin{theorem}\label{shifted primes intro version}
Let $a$ be a fixed nonzero integer. Under the hypotheses of Theorem \ref{main theorem intro version}, the function $Q(g_1(p-a),\dots,g_\ell(p-a))$ satisfies an Erd\H os--Kac law with $A_Q(x)$ and $B_Q(x)$ as defined in equations~\eqref{AQx intro} and~\eqref{BQx intro}. In other words, for every real number $u$,
\begin{equation} \label{EK limit shifted primes}
\lim_{x\to\infty} \frac1{\pi(x)} \# \big\{ p\le x\colon Q(g_1(p-a),\dots,g_\ell(p-a)) < A_Q(x) + u B_Q(x) \big\} = \frac1{\sqrt{2\pi}} \int_{-\infty}^u e^{-t^2/2}\,dt.
\end{equation}
\end{theorem}

In the next section we describe the generalized setting, with arbitrary finite sets $\P$ and $\A$ as mentioned above and other notation used throughout the paper, and establish a general distributional limit law (Theorem~\ref{main theorem sieve version}) from the calculation of the appropriate $m$th moments (Theorem~\ref{moment theorem}). Further, in that section, we also derive Theorems~\ref{main theorem intro version} and~\ref{shifted primes intro version} from Theorem~\ref{main theorem sieve version}. The remainder of the paper will then be devoted to establishing Theorem~\ref{moment theorem}, other than giving some examples in Section~\ref{examples sec} that in particular establish Corollaries~\ref{cor1} and~\ref{cor2}.

\section{Notation and restatement of main theorem}  \label{notation section}

Following the methods established in \cite{gs07}, we begin by setting up a sieve-theoretic framework within which we will recast Theorem~\ref{main theorem intro version}.


\begin{definition} \label{A and E def}
Let $\A$ denote a finite set (or multiset) of positive integers.
For every positive integer $d$, set $\A_d = \{ a\in\A \colon d\mid a \}$. As is usual in sieve theory, we suppose that there exists a positive real number $X$ that is a good approximation to the cardinality of~$\A$, and a multiplicative function $h(d)$ (satisfying $0 \le h(d) \le d$) such that $\frac{h(d)}{d} X$ is a good approximation to $\#\A_d$.
More explicitly, we define the remainder terms $E_d$ by 
\begin{align}\label{sieveequation}
\#\A_d = \frac{h(d)}{d} X + E_d,
\end{align}
and we suppose that we have some suitable control over $E_d$, individually or on average. It would certainly suffice for our purposes to have $E_d \ll d^c$ for some fixed $c<1$.
\end{definition}

In the setting of Theorem~\ref{main theorem intro version}, for example, we will have $E_d \ll 1$. In general, one needs only to know the behavior of the terms $E_d$ on average; for example, in the proof of Theorem~\ref{shifted primes intro version} we will handle the terms $E_d$ via the Bombieri--Vinogradov theorem.

\begin{definition} \label{mu sigma def}
Let $\P$ denote a finite set of primes.
For any nonnegative integer $m$, define
\begin{align*}
\D_k(\P) &= \bigcup_{\ell=0}^k \{ p_1\cdots p_\ell\colon p_1,\dots,p_\ell \text{ are distinct elements of } \P \} ;
\end{align*}
note that the elements of $\D_k(\P)$ are squarefree integers contained in the interval $[1,(\max \P)^k]$. In particular, $\D_0(\P)=\{1\}$.

For any function $g\colon\N\to\R$ (but usually for strongly additive functions $g$), we write
\[
g^\P(a) = \sum_{\substack{p \in \P \\ p \mid a}} g(p).
\]
(We remark that this notation would be sensible even if $\P$ were an infinite set of primes, since any given integer $a$ has only finitely many prime factors.)
In keeping with the terminology and notation established in \cite{gs07}, we define the {\em mean} of $g$ over $\P$ as
\begin{equation}  \label{mean}
\mu^\P(g) = \sum_{p\in\P} g(p) \frac{h(p)}p;
\end{equation}
for example, $\mu^\P(1) =  \sum_{p\in\P} {h(p)}/p$.
We continue to borrow from probabilistic terminology by defining the {\em standard deviation} of $g$ over $\P$ as
\[
\sigma^\P(g) = \bigg( \sum_{p \in \P} g(p)^2 \frac{h(p)}p \bigg(1 - \frac{h(p)}{p} \bigg) \bigg)^{1/2},
\]
and we call the quantity $\sigma^\P(g)^2$ the {\em variance} of $g$ over~$\P$.
Furthermore, for two functions $g_1,g_2\colon\N\to\R$ we define the {\em covariance} of $g_1$ and $g_2$ over $\P$ by
\[
\kappa^\P(g_1,g_2) = \sum_{p \in \P} g_1(p)g_2(p) \frac{h(p)}p\bigg( 1-\frac{h(p)}{p} \bigg).
\]
Note that $\sigma^\P(g)^2 = \kappa^\P(g, g)$.
\end{definition}

It is perhaps worth emphasizing that here the set $\P$ represents the ``good'' primes whose contribution to the values of $g(n)$ we want to include, for all integers $n\in\A$; this is as opposed to sieve theory, where one usually names the set of ``bad'' primes to exclude, for the goal of seeking special elements of $\A$ free of such prime factors.

For example, consider the case where $\P$ is the set of all primes up to $x$, and $g(n)=\omega(n)$. A result of Mertens then says that the mean $\mu^\P(\omega) \sim \log\log x$ and the variance $\sigma^\P(\omega)^2 \sim \log\log x$, which are the quantities that we recognize from the classical Erd{\H o}s--Kac theorem.

In the following few definitions, the objects being defined will depend upon $\ell$ functions $g_1,\dots,g_\ell\colon\N\to\R$; we suppress the dependence upon these functions from the notation.

\begin{definition} \label{mf def}
We let $\mfK$ denote the maximum variance among the functions $g_j$:
\[
\mfK = \max_{1 \leq i,j \leq \ell} | \kappa^\P(g_i,g_j) | = \max_{1 \leq i \leq \ell} \sigma^\P(g_i)^2.
\]
The latter equality is justified by the Cauchy--Schwarz inequality, which implies that
\[
| \kappa^\P(g_i,g_j) | \leq \max\{\sigma^\P(g_i)^2, \sigma^\P(g_j)^2\}
\]
Similarly, we let $\mfM$ denote the maximum mean among the functions $g_j$:
\[
\mfM = \max_{1 \leq i \leq \ell} \mu^\P(g_i).
\]
Note that if there exists a positive constant $G$ such that $0\le g_j(p) \le G$ for all $1\le j\le \ell$ and all $p\in\P$, then trivially $\mfK \leq G \mfM$.
\end{definition}

\begin{definition}
Suppose that the functions $g_1,\dots,g_\ell$ take nonnegative values.
Given any polynomial $Q(T_1,\dots,T_\ell)$ in $\ell$ variables with positive coefficients, we define
\begin{equation}  \label{AQP def}
A_Q(\P) = Q(\mu^{\P}(g_1), \ldots, \mu^{\P}(g_\ell))
\end{equation}
and
\begin{equation}  \label{BQP def}
B_Q(\P) = \bigg( \sum_{i=1}^\ell \sum_{j=1}^\ell \frac{\partial Q}{\partial T_i} \big( \mu^{\P}(g_1),\dots,\mu^{\P}(g_\ell) \big) \frac{\partial Q}{\partial T_j} \big( \mu^{\P}(g_1),\dots,\mu^{\P}(g_\ell) \big) \kappa^{\P}(g_i,g_j) \bigg)^{1/2}.
\end{equation}
\end{definition}

\begin{definition} \label{Mh def}
For any polynomial $Q(t_1,\dots,t_\ell)$ in $\ell$ variables and any nonnegative integer~$m$, define the $m$th moment
\[
{M}_m = \sum_{a\in\A} \big( Q(g_1^{{\P}}(a),\dots,g_\ell^{{\P}}(a)) - A_Q(\P) \big)^m.
\]
\end{definition}

\begin{definition} \label{Cm def}
For an integer $m \geq 0$, let
\[
C_m = \begin{cases} \displaystyle \frac{m!}{2^{m/2}(m/2)!}, & \text{if } m \text{ is even}, \\ 0, & \text{if } m \text{ is odd} \end{cases}
\]
denote the $m$th moment of the standard normal distribution.
\end{definition}

These last two definitions foreshadow our strategy of deducing a normal limiting distribution from asypmtotic formulas for the $m$th moments of a centered version of our quantity of interest. Indeed, under all of the notation above, we establish the following formula for the~${M}_m$:

\begin{theorem}\label{moment theorem}
Let $Q(T_1, \ldots,T_\ell)$ be a polynomial of degree $\delta$ with nonnegative real coefficients, let $m$ be a positive integer, and let $\P$ be a finite set of primes. Let $g_1, g_2, \ldots, g_\ell$ be nonnegative strongly additive functions such that, for some fixed $G > 0$, we have $g_j(p) \leq G$ for each $j$ and for all primes $p \in {\P}$.
Let $\A$ be a finite multiset of positive integers such that, in the notation of Definition~\ref{A and E def},
\begin{equation}\label{stronger error bound}
\mu^\P(1)^{\delta m} \sum_{d\in \D_{\delta m}(\P)} |E_d| \ll X \mfK^{\delta m/2 - 1}
\end{equation}
If $m$ is even, then
\[
{M}_m = C_m X B_Q(\P)^m + O(X \mfM^{\delta m - (m + 1)/2} ),
\]
while if $m$ is odd, then
\[
{M}_m \ll X B_Q(\P)^{m-1} + X \mfM^{\delta m - (m + 1)/2}.
\]
Here, the implied constants may depend on the polynomial $Q$, the moment index $m$, the constant $G$, and the finite set $\A$.
\end{theorem}

Proving this theorem is the primary goal of this paper, and we will begin the proof in the next section. For the remainder of this section, however, we expound upon the consequences of Theorem~\ref{moment theorem} to various Erd{\H o}s--Kac laws.

Let $\tilde{\A}$ denote an infinite set of positive integers, and define
\begin{equation}  \label{remainder assumption infinite}
\A(x) = \{a \in \tilde{A} \colon a \leq x \}
\end{equation}
(indeed, $\tilde{\A}$ may even be a multiset, as long as ${\A}(x)$ is finite for every~$x$).
Furthermore, let $\tilde{\P}$ denote an infinite set of primes, and define $\P(z) = \{p \in \tilde{\P} \colon p \leq z\}$.
Using Theorem \ref{moment theorem}, we can prove the following result, which (as we subsequently show) implies Theorems~\ref{main theorem intro version} and~\ref{shifted primes intro version}.

\begin{theorem}\label{main theorem sieve version}
Let $Q(T_1, \ldots, T_\ell)$ be a polynomial of degree $\delta$ with nonnegative real coefficients. Let $\tilde{\A}$ and $\tilde{\P}$ be infinite sets as described above, and assume that there exists a positive constant $\eta$ such that
\[
\mu^{\P(t)}(1) = \sum_{p\le t} \frac{h(p)}p \sim \eta\log\log t.
\]
Let $x$ and $z=z(x)$ be parameters tending to infinity, with $z$ chosen so that for all integers $k\ge0$,
\begin{equation}  \label{main theorem sieve version remainder}
\mu^{\P(z)}(1)^{k} \sum_{d\in \D_{k}(\P(z))} |E_d| \ll_k X \mfK^{k/2 - 1},
\end{equation}
where $\mfK$ is as in Definition~\ref{mf def} with $\P$ replaced by~$\P(z)$.
Let $g_1, g_2, \ldots, g_\ell$ be nonnegative strongly additive functions such that for each $1\le j\le\ell$:
\begin{enumerate}
\item $g_j(p)\ll1$ uniformly for all primes $p \in \tilde{\P}$;
\item the series $\sum_{p\in\tilde\P} {g_j^2(p)}/p$ diverges;
\item $\sigma^{\P(z)}(g_j)^2 \gg \mfM$, with $\mfM$ as in Definition~\ref{mf def} with $\P$ replaced by $\P(z)$.
\end{enumerate}
Further, assume that
\begin{equation}  \label{log x log z little o}
\frac{\log x}{\log z} = o( \mfM^{1/2} ).
\end{equation}

Then the values $Q(g_1(a), \ldots, g_\ell(a))$, as $a$ runs through the elements of $\A$, satisfy an Erd{\H o}s--Kac law with mean $A_Q(\P(x))$ from equation~\eqref{AQP def}
and variance $B_Q(\P(x))^2$ from equation~\eqref{BQP def};
in other words, for every $u\in\R$,
\begin{equation} \label{main theorem sieve version conclusion}
\lim_{x\to\infty} \frac{\# \big\{a\in\A(x) \colon Q\big(g_1(a), \ldots, g_\ell(a)\big) < A_Q(\P(x)) + u B_Q(\P(x)) \big\}}{\#\A(x)} = \frac1{\sqrt{2\pi}} \int_{-\infty}^u e^{-t^2/2}\,dt.
\end{equation}
\end{theorem}

Before we show that Theorem~\ref{main theorem sieve version} is a consequence of Theorem~\ref{moment theorem}, we make the following convention regarding constants implicit in $O$, $\ll$, and $\asymp$ notation: these constants may always depend upon the polynomial $Q$ and its number of arguments (usually $\ell$), the pointwise bound~$G$ for the additive functions $g_j$, the integer $m$ dictating which moment we are looking at, and the finite sets $\A$ and $\P$ where appropriate. (We remind the reader that we write $f \asymp g$ if there exist absolute constants $C > c > 0$ such that $c  f(x) \leq g(x) \leq C  f(x)$ throughout the domain of $f$ and~$g$.) On the other hand, these implicit constants will never depend upon the parameters $x$ and $z$, and therefore will be independent of all quantities that depend upon these parameters, such as $M_m$, $\mu^\P(g_j)$ and $\mfM$, and $\kappa^\P(g_i,g_j)$ and~$\mfK$.

\begin{proof}[Deduction of Theorem~\ref{main theorem sieve version} from Theorem~\ref{moment theorem}.]
For this proof, we will set $\A=\A(x)$ and $\P = \P(z)$ (and will write out $\P(x)$ explicitly when necessary). We begin by noting some consequences of the assumptions of Theorem~\ref{main theorem sieve version}.
First, thanks to assumption~(a), there exists a constant $G>0$ such that $g_j(p) \leq G$ for each $j$ and for all primes $p\in\tilde{\P}$. In particular, all of the hypotheses of Theorem~\ref{moment theorem} are satisfied for any positive integer~$m$.

Next, for each $1\le j\le\ell$ we have $\mfM \ge \mu^\P(g_j) \ge \frac1G \sigma^\P(g_j)^2 \gg \mfM$ by assumption~(c),
so that each $\mu^\P(g_j) \asymp \mfM$ and hence $\mfK \asymp \mfM$; furthermore, $\mfM \to \infty$ as $x \to \infty$ thanks to assumption~(b). In summary, for every $1\le j\le\ell$,
\begin{equation} \label{all asymp}
\sigma^\P(g_j)^2 \asymp \mu^\P(g_j) \asymp \mfK \asymp \mfM \to \infty.
\end{equation}
Furthermore, by Cauchy--Schwarz, we also have $0 \le \kappa^\P(g_i,g_j) \ll \mfM$.
Note also that when $m=0$, the assumption~\eqref{main theorem sieve version remainder} becomes simply $|E_1| \ll \frac X\mfK$; this implies that $\#\A = \#\A_1 = \frac{h(1)}1X+E_1 = X (1+O(\frac1\mfK)) \sim X$ as $x\to\infty$.

Since $Q$ is of degree $\delta$, there exists $1\le j_0\le\ell$ such that
$\frac{\partial Q}{\partial T_{j_0}} ( T_1,\dots,T_\ell )^2$
is of degree $2\delta - 2$. It follows from the nonnegativity of the covariances and of the coefficients of $Q$ that
\begin{align}
B_Q(\P)^2 &= \sum_{i=1}^\ell \sum_{j=1}^\ell \frac{\partial Q}{\partial T_i} \big( \mu^\P(g_1),\dots,\mu^\P(g_\ell) \big) \frac{\partial Q}{\partial T_j} \big( \mu^\P(g_1),\dots,\mu^\P(g_\ell) \big) \kappa^\P(g_i, g_j) \notag \\
&\ge \frac{\partial Q}{\partial T_{j_0}} \big( \mu^\P(g_1),\dots,\mu^\P(g_\ell) \big)^2 \kappa^\P(g_{j_0},g_{j_0}) \gg \mfM^{2\delta - 2} \sigma^\P(g_{j_0})^2 \gg \mfM^{2\delta - 1} \label{B lower bound}
\end{align}
by equation~\eqref{all asymp}. If $m$ is even, then Theorem~\ref{moment theorem} (with $\A=\A(x)$ and $\P = \P(z)$) implies
\[
\frac{{M}_m}{XB_Q(\P)^m} = C_m + O(\mfM^{-1/2}) = C_m + o(1);
\]
on the other hand, if $m$ is odd, then Theorem~\ref{moment theorem} implies
\[
\frac{{M}_m}{XB_Q(\P)^m} \ll B_Q(\P)^{-1} + \mfM^{-1/2} \ll \mfM^{-2\delta+1} + \mfM^{-1/2} = o(1).
\]
These estimates establish the limits
\begin{equation} \label{limits needed}
\lim_{x \to \infty} \frac{{M}_m}{\#\A \cdot B_Q(\P)^m} = \lim_{x \to \infty} \frac{{M}_m}{XB_Q(\P)^m} = \begin{cases}
C_m, &\text{if $m$ is even}, \\
0, &\text{if $m$ is odd}, \end{cases}
\end{equation}
where $M_m$ is as in Definition~\ref{Mh def}; the first equality is due to the fact that $\#\A \sim X$ as $x\to\infty$.
Note that $M_0 = \#\A$, and so the limit~\eqref{limits needed} is trivially true when $m=0$ as well.


It follows from the limits~\eqref{limits needed}, by the method of moments (in a way that is standard in these applications to Erd\H os--Kac theorems; see~\cite[Section 7]{mt18} for more details on this type of deduction), that
\begin{equation} \label{not quite right conclusion}
\lim_{x\to\infty} \frac{\# \big\{a\in\A(x) \colon Q\big(g_1^\P(a), \ldots, g_\ell^\P(a)\big) < A_Q(\P) + u B_Q(\P) \big\}}{\#\A(x)} = \frac1{\sqrt{2\pi}} \int_{-\infty}^u e^{-t^2/2}\,dt,
\end{equation}

Note that this deduction is not exactly the same as the conclusion of the theorem we are proving: we would rather that the polynomial $Q$ were being evaluated at the original additive functions $g_j(a)$ rather than their truncations $g_j^\P(a)$, and that both occurrences of $\P=\P(z)$ on the right-hand side of the inequality were instead $\P(x)$. However, it is easy to see (essentially by the continuity of the integral as a function of~$u$) that we can make these adjustments provided that:
\begin{enumerate}
\item[(i)] $A_Q(\P(x)) - A_Q(\P(z)) = o\big( B_Q(\P(z)) \big)$;
\item[(ii)] $B_Q(\P(x)) - B_Q(\P(z)) = o\big( B_Q(\P(z)) \big)$;
\item[(iii)] $Q\big(g_1(a), \ldots, g_\ell(a)\big) - Q\big(g_1^{\P(z)}(a), \ldots, g_\ell^{\P(z)}(a)\big) = o\big( B_Q(\P(z)) \big)$ for every $a\in\A(x)$.
\end{enumerate}
First, note that $A_Q(\P(z)) \asymp \mfM^\delta$ by equations~\eqref{AQP def} and~\eqref{all asymp}. Furthermore, equation~\eqref{B lower bound} gives the lower bound $B_Q(\P(z)) \gg \mfM^{\delta-1/2}$, and the corresponding upper bound follows from equation~\eqref{all asymp} and the estimate $\kappa^\P(g_i,g_j) \ll \mfM$.
Also note that for each $1\le j\le\ell$,
\begin{align*}
0 \le \mu^{\P(x)}(g_j) - \mu^{\P(z)}(g_j) = \sum_{z<p\le x} g_j(p) \frac{h(p)}p \le G \sum_{z<p\le x} \frac{h(p)}p = G\eta \log \frac{\log x}{\log z} + O(1) \ll \log \mfM
\end{align*}
by the assumption~\eqref{log x log z little o}; by a similar calculation, $\kappa^\P(g_i,g_j) \ll \log \mfM$ for all $1\le i,j\le \ell$.
The difference
\[
A_Q(\P(x)) - A_Q(\P(z)) = Q\big( \mu^{\P(x)}(g_1), \dots, \mu^{\P(x)}(g_\ell) \big) - Q\big( \mu^{\P(z)}(g_1), \dots, \mu^{\P(z)}(g_\ell) \big)
\]
can therefore be bounded, using the multivariable mean value theorem, by $\mfM^{\delta-1} \log \mfM = o(\mfM^{\delta-1/2}) = o\big( B_Q(\P(z)) \big)$, establishing~(i).

An analogous argument shows that the difference $B_Q(\P(x))^2 - B_Q(\P(z))^2$ can be bounded by $\mfM^{2\delta-2} \log \mfM$; since
\[
B_Q(\P(x)) - B_Q(\P(z)) = \frac{B_Q(\P(x))^2 - B_Q(\P(z))^2}{B_Q(\P(x)) + B_Q(\P(z))} \asymp \frac{B_Q(\P(x))^2 - B_Q(\P(z))^2}{\mfM^{\delta-1/2}},
\]
we see that $B_Q(\P(x)) - B_Q(\P(z)) \ll \mfM^{\delta-3/2} \log \mfM = o\big( B_Q(\P(z)) \big)$, establishing~(ii).

Finally, for any $a\in\A(x)$ and any $1\le j\le\ell$,
\[
g(a) - g^{\P(z)}(a) = \sum_{\substack{z<p\le x \\ p\mid a}} g(p)
\]
since $a\le x$; the number of summands in this sum is at most $\frac{\log x}{\log z}$, whence $g^{\P(x)}(a) - g^{\P(z)}(a) \le G \frac{\log x}{\log z} = o(\mfM^{1/2})$ by the assumption~\eqref{log x log z little o}. Again the multivariable mean value theorem gives the estimate $Q\big(g_1(a), \ldots, g_\ell(a)\big) - Q\big(g_1^{\P(z)}(a), \ldots, g_\ell^{\P(z)}(a)\big) \ll \mfM^{\delta-1} \cdot o(\mfM^{1/2}) = o\big( B_Q(\P(z)) \big)$, establishing~(iii).

These estimates confirm that we can deduce equation~\eqref{main theorem sieve version conclusion} from equation~\eqref{not quite right conclusion}, which concludes the proof of Theorem~\ref{main theorem sieve version}.
\end{proof}

Theorem \ref{main theorem intro version} follows quickly from Theorem \ref{main theorem sieve version}.

\begin{proof}[Deduction of Theorem~\ref{main theorem intro version} from Theorem~\ref{main theorem sieve version}.]
Let $\tilde{\P}$ be the set of all primes and $\tilde{\A}$ the set of all positive integers; then $\#\A_d = \frac{x}{d} + O(1)$, and so setting $X=x$ and $h(d)=1$ for all $d\ge1$ yields $|E_d| \ll 1$, as well as $\mu^{\P(t)}(1) = \sum_{p\le t} \frac1p \sim \log\log t$. Let the polynomial $Q(T_1, \ldots, T_\ell)$ and the additive functions $g_1, \ldots, g_\ell$ satisfy the hypotheses of Theorem~\ref{main theorem intro version}; note that one of those hypotheses, namely $\kappa(g_i,g_i) \gg \mu(g_j)$ (with the implicit parameter $x$ replaced by $z$), implies that $\sigma^{\P(z)}(g_j)^2 \gg \mfM$.

Let $z=z(x)$ be the smallest positive real number satisfying $z^{\max\{1,\mfM^{1/3}\}} \ge x$, with $\mfM$ as in Definition~\ref{mf def} with $\P$ replaced by~$\P(z)$.
Since $\mfM^{1/3}$ is an increasing, right-continuous function of $z$, this minimum is well-defined and tends to infinity with~$x$. Also, $z \ll_\ep x^\ep$ for any $\ep>0$ since $\mfM$ tends to infinity as well (again because $\sum_{p\le z} \frac{g_j(p)}p \ge \frac1G\sum_{p\le z} \frac{g_j(p)^2}p$, with the latter series diverging by assumption as $z\to\infty$).
We verify quickly that the hypothesis~\eqref{main theorem sieve version remainder} is satisfied for all $k \ge 0$ via the calculation
\begin{align*}
\mu^{\P(z)}(1)^{k} \sum_{d \in \D_{k}(\P(z))} |E_d| &\ll (\log\log x)^{k} \# \D_{k}(\P(z)) \\ &\le (\log\log x)^{k} z^{k} \ll_{\ep} x^\ep \ll_\ep \frac X{\log\log x} \ll X\mfK^{k/2-1},
\end{align*}
since each $\kappa^{\P(z)}(g_i,g_j) \le \kappa^{\P(x)}(g_i,g_j) \le \sum_{p\le x} \frac{G^2}p(1-\frac1p) \ll \log\log x$. We also confirm that $\frac{\log x}{\log z} = \max\{1,\mfM^{1/3}\} = o(\mfM^{1/2})$, which completes the verification of the hypotheses of Theorem~\ref{main theorem sieve version}.

Therefore the conclusion of Theorem~\ref{main theorem sieve version} holds, and it is easy to see that this is identical to the conclusion of Theorem~\ref{main theorem intro version} since $\mu^{\P(x)}(g_j) = \mu(g_j)$ and $\kappa^{\P(x)}(g_i,g_j) = \kappa(g_i,g_j)$.
\end{proof}

Theorem \ref{shifted primes intro version} follows in almost exactly the same way, other than requiring a more powerful tool in the Bombieri--Vinogradov theorem to control the accumulation of the sieve error terms. See for example~\cite[Theorem 17.1]{ik04} for the statement of the Bombieri--Vinogradov theorem for the function $\psi(x;q,a)$, from which it is simple to derive the analogous version for $\pi(x;q,a)$ (an example of such a derivation is the proof of~\cite[Corollary 1.4]{AH}):

\begin{prop}\label{bombvino}
For any positive real number $A$, there exists a positive real number $B=B(A)$ such that for all $x\ge2$,
\begin{align}
\sum_{2\le q \le x^{1/2}(\log x)^{-B}} \max_{(a, q) = 1} \left| \pi(x; q, a) - \frac{\li(x)}{\varphi(q)} \right| &\ll_A \frac{x}{(\log x)^A}, \label{pi BV}
\end{align}
where $\li(x) = \int_2^x \frac{dt}{\log t}$ is the usual logarithmic integral.
\end{prop}

\begin{proof}[Deduction of Theorem~\ref{shifted primes intro version} from Theorem~\ref{main theorem sieve version}.]

Let $\tilde{\P}$ be the set of all prime numbers, and set $\tilde{\A} = \{p - a \colon p \text{ prime},\, p>a\}$. Define a (strongly) multiplicative function $h(d)$ by setting $h(p) = \frac p{p-1}$ if $p \nmid a$ and $h(p) = 0$ if $p \mid a$. Then, with $X=\li(x)\sim\pi(x)$, we have
\[
\#\A_d = \pi(x; d, a) = \frac{h(d)}{d}\li(x) + E_d(x)
\]
where, if $(a, d) = 1$, then $E_d(x) = \pi(x; d, a) - \frac{\li(x)}{\varphi(d)}$ is the error term in the prime number theorem for arithmetic progressions, while if $(a, d) > 1$, then $|E_d(x)| \le 1$. Therefore, for any parameter $z=x^{o(1)}$,
%
%
\[
\mu^\P(1)^{k} \sum_{d \in \D_{k}(\P)} |E_d(x)| \le \bigg( \sum_{p\le z} \frac1{p-1} \bigg)^{k} \sum_{d \le z^{k}} |E_d(x)|
 \ll (\log\log x)^{k} \frac{x}{(\log x)^2} \ll X \mfK^{k/2 - 1}
\]
by the Bombieri--Vinogradov bound~\eqref{pi BV}.
The rest of the proof is the same as the deduction of Theorem~\ref{main theorem intro version} from Theorem~\ref{main theorem sieve version} above.
\end{proof}

From now on, we concern ourselves entirely with the proof of Theorem~\ref{moment theorem}; in particular, we assume for the rest of this paper that the finite sets $\A$ and $\P$ and the additive functions $g_1, \ldots, g_\ell$ satisfy all the hypotheses of Theorem~\ref{moment theorem}.

\section{Polynomial accounting}

We attack the $m$th moment $M_m$ by expanding the $m$th power in each summand. In this section we present the system we use for writing down this expansion, starting by defining a few helpful objects and pieces of notation. For any given $a\in\A$, define $f_r(a)$ to be the completely multiplicative function of $r$ (not of $a$) that satisfies
\begin{equation}  \label{f def}
f_p(a) = \begin{cases} 1-h(p)/p, &\text{if } p\mid a, \\ -h(p)/p, &\text{if } p \nmid a.\end{cases}
\end{equation}
(We note that this function was first introduced by Granville and Soundararajan in \cite{gs07}, where it served largely the same role as it will for us.) For any strongly additive function $g$, define the notation
\begin{equation} \label{F_g def}
F_g^\P(a) = \sum_{p\in\P} g(p) f_p(a),
\end{equation}
and note that equation~\eqref{mean} implies that $g^\P(a) = \mu^\P(g) + F^\P_g(a)$. Therefore, from Definition~\ref{Mh def} and equation~\eqref{AQP def},
\begin{align}\label{Mm with mus and Fs}
M_m &= \sum_{a\in\A} \big( Q(g_1^{{\P}}(a),\dots,g_\ell^{{\P}}(a)) - A_Q(\P) \big)^m \nonumber \\
&= \sum_{a\in\A} \big( Q(F_{g_1}^\P(a)+\mu^\P(g_1),\dots,F_{g_\ell}^\P(a)+\mu^\P(g_\ell)) - Q(\mu^\P(g_1),\dots,\mu^\P(g_\ell)) \big)^m.
\end{align}

We are now ready to expand the $m$th power. In addition to setting out the necessary notation for writing down this expansion, this section establishes a key result, Proposition~\ref{Rh magic Phi lemma}, which allows us to simplify the main term of the moments $M_m$ into the form of Theorem~\ref{moment theorem}. We note that this part of our proof is extremely similar to~\cite[Section~4]{mt18}; indeed, we need only quote two relevant definitions and two relevant results from that section, beginning with~\cite[Definition~4.1]{mt18}:

\begin{definition} \label{Tk def}
%
For any positive even integer $k$, define $T_k$ to be the set of all 2-to-1 functions from $\{1,\dots,k\}$ to $\{1,\dots,k/2\}$. A typical element of $T_k$ will be denoted by $\tau$.
For $\tau\in T_k$ and $j\in\{1,\dots,k/2\}$, define $\Upsilon_1(j)$ and $\Upsilon_2(j)$ to be the two preimages in $\{1,\dots,k\}$; we will never be in a situation where we need to distinguish them from each other.
\end{definition}

We also quote~\cite[Definition~4.5]{mt18}:

\begin{definition} \label{Rh def}
Given a polynomial $Q(y_1,\dots,y_\ell) \in \R[y_1,\dots,y_\ell]$ of degree $\delta$, and a positive integer $m$, define a polynomial in the $2\ell$ variables $x_1,\dots,x_\ell,y_1,\dots,y_\ell$ by
\[
R_m(x_1,\dots,x_\ell,y_1,\dots,y_\ell) = \big( Q(x_1+y_1,\dots,x_\ell+y_\ell) - Q(y_1,\dots,y_\ell) \big)^m.
\]
To expand this out in gruesome detail, $R_m$ can be written as the sum of $B_m$ monomials, the $\beta$th of which will have total $x$-degree equal to $k_{m\beta}$ and total $y$-degree equal to $\tilde k_{m\beta}$:
\begin{equation} \label{Rh expanded}
R_m(x_1,\dots,x_\ell,y_1,\dots,y_\ell) = \sum_{\beta=1}^{B_m} r_{m\beta} \prod_{i=1}^{k_{m\beta}} x_{v(m,\beta, i)} \prod_{j=1}^{\tilde k_{m\beta}} y_{w(m,\beta, j)}.
\end{equation}
Here each $v(m,\beta, i)$ and $w(m,\beta, j)$ is an integer in $\{1,\dots,\ell\}$. Note that the total $x$-degree of the $\beta$th monomial in the sum is $k_{m\beta}$, while its total $y$-degree is $\tilde k_{m\beta}$. We note that, in particular, each $k_{m\beta}$ is at least $m$. Note also that each $k_{m\beta}+\tilde k_{m\beta}$ is at most $\delta m$, and that there exists at least one $\beta$ for which $k_{m\beta}=\tilde k_{m\beta}=m$.
\end{definition}

\begin{remark}
Comparing Definition \ref{Rh def} with the expression (\ref{Mm with mus and Fs}), in our calculation we will take $x_i = F^\P_{g_i}(a)$ and $y_i = \mu^\P(g_i)$, for each $1 \leq i \leq \ell$. Thus the right-hand side of equation (\ref{Rh expanded}) becomes a sum of products of $\mu^\P(g_i)$ and $F^\P_{g_j}(a)$ for various choices of $1 \leq i, j \leq \ell$.
\end{remark}

Note that the monomial in $R_m$ of smallest $x$-degree has $x$-degree equal to $m$. It turns out that these terms contribute to the main term of the moments $M_m$. The following proposition will allows us to prove this fact; it appears as~\cite[Proposition 4.7]{mt18}, and its proof can be found there. We use $Q_j(y_1, \ldots, y_\ell)$ to denote the partial derivative of $Q(y_1,\dots,y_\ell)$  with respect to the $j$th variable.

\begin{prop} \label{Rh magic Phi lemma}
Let $m$ be a positive even integer, and let $z_{ij}$ ($1\le i,j\le\ell$) be real numbers. In the notation of Definitions~\ref{Tk def} and~\ref{Rh def},
\begin{multline} \label{after applying Phi forms}
\frac1{(m/2)!} \sum_{\substack{\beta\le B_m \\ k_{m\beta}=m}} r_{m\beta} \prod_{j=1}^{\tilde k_{m\beta}} y_{w(m,\beta, j)} \sum_{\tau\in T_m} \prod_{i=1}^{m/2} z_{v(m,\beta,\Upsilon_1(i))v(m,\beta,\Upsilon_2(i))} \\
= C_m \bigg( \sum_{i=1}^\ell \sum_{j=1}^\ell Q_i(y_1,\dots,y_\ell) Q_j(y_1,\dots,y_\ell) z_{ij} \bigg)^{m/2},
\end{multline}
where $C_m$ is as in Definition~\ref{Cm def}.
\end{prop}

\section{Preliminary calculations}

As we have seen, expanding the $m$th power in $M_m$ will result in a sum of terms, each of which takes the shape of a product of means $\mu^\P(g_i)$ and functions $F^\P_{g_j}(a)$ for various choices $1 \leq i, j \leq \ell$. In this section, we obtain formulas for these products. When $m$ is even, those terms which contribute to the main term of the $m$th moment will involve a product of $m/2$ of the functions $F^{\P}_{g_i}(a)$. Proposition \ref{summing products of k Fs} allows us to write such products in terms of covariances $\kappa(g_i, g_j)$, thereby (after some rearranging) recovering, via Proposition~\ref{Rh magic Phi lemma}, the main term stated in Theorem~\ref{moment theorem}. We also obtain upper bounds for those terms not contributing to this main term. Our method here is deeply inspired by the work of Granville and Soundararajan in~\cite{gs07}.

We begin by defining two new functions which serve important technical roles (these functions were called $G(n)$ and $E(r,s)$ in~\cite{gs07}, but we have renamed them to avoid a clash of notation herein).

\begin{definition}\label{H and E definition}
Let $H(n)$ be the multiplicative function defined by
\[
H(n) = \prod_{p^\alpha \| n} \bigg( \frac{h(p)}p \bigg( 1-\frac{h(p)}p \bigg)^\alpha + \bigg( {-}\frac{h(p)}p \bigg)^\alpha \bigg( 1-\frac{h(p)}p \bigg) \bigg).
\]
For a given natural number $s$, let $J(r,s)$ be the multiplicative function of $r$ defined by
\[
J(r,s) = \prod_{\substack{p^\alpha\| r \\ p\mid s}} \bigg( \bigg( 1-\frac{h(p)}p \bigg)^\alpha - \bigg( {-}\frac{h(p)}p \bigg)^\alpha \bigg) \prod_{\substack{p^\alpha\| r \\ p\nmid s}} \bigg( {-}\frac{h(p)}p \bigg)^\alpha.
\]
\end{definition}

\noindent The following two lemmas list some properties of these two functions that will be useful in later proofs.

\begin{lemma} \label{simple H lemma}
For all primes $p$ and all positive integers $m$, $s$, and~$\alpha$:
\begin{enumerate}
\item $H(p)=0$, and $H(n)=0$ unless $n$ is squarefull.
\item $|H(p^\alpha)| \le H(p^2)$.
\item $|J(p^\alpha,s)|\le1$; furthermore, if $p\nmid s$ then $|J(p^\alpha,s)|\le h(p)/p$.
\end{enumerate}
\end{lemma}

\begin{proof}
Part (a) is obvious. As for part (b), the triangle inequality yields
\[
|H(p^\alpha)| \le \frac{h(p)}p \bigg( 1-\frac{h(p)}p \bigg)^\alpha + \bigg( \frac{h(p)}p \bigg)^\alpha \bigg( 1-\frac{h(p)}p \bigg).
\]
Since $0 \le h(p)/p \le 1$, both terms on the right-hand side are decreasing functions of $\alpha$, and so
\begin{equation} \label{Hp^2}
|H(p^\alpha)| \le \frac{h(p)}p \bigg( 1-\frac{h(p)}p \bigg)^2 + \bigg( \frac{h(p)}p \bigg)^2 \bigg( 1-\frac{h(p)}p \bigg) = \frac{h(p)}{p} \bigg( 1 - \frac{h(p)}{p} \bigg) = H(p^2).
\end{equation}
Finally, we consider
\[
J(p^\alpha, s) = \begin{cases} \left( 1-{h(p)}/p \right)^\alpha - \left( {-}{h(p)}/p \right)^\alpha, &\mbox{if } p \mid s, \\ 
\left( -{h(p)}/{p} \right)^\alpha, & \mbox{if } p \nmid s. \end{cases}
\]
As $h(p) \leq p$, the second assertion of part (c) is immediate; and, by the triangle inequality, if $p \mid s$ then
\[
|J(p^\alpha, s)| \leq \left| 1 - \frac{h(p)}{p} \right|^\alpha + \left|\frac{h(p)}{p}\right|^\alpha.
\]
The right-hand side is a decreasing function of $\alpha \geq 1$ which takes the value 1 if $\alpha = 1$, which establishes the last remaining assertion.
\end{proof}

\begin{lemma}  \label{bonus lemma}
Let $r$ be a positive integer and let $R$ be the largest squarefree divisor of~$r$. Then for any divisor $s$ of~$R$,
\[
\sum_{d \mid R} f_r(d) \frac{h(d)}{d} \prod_{p \mid R/d} \bigg( 1 - \frac{h(p)}{p} \bigg) = H(r)
\quad\text{and}\quad
\sum_{de=s} f_r(d)\mu(e) = J(r, s).
\]
\end{lemma}

\begin{proof}
For each prime factor $p$ of $r$, let $\alpha_p$ denote the exponent of $p$ in the factorization of~$r$, so that $p^{\alpha_p} \| r$. We start with the first claimed equation. By Definition~\ref{H and E definition},
\[
H(r) = \prod_{p \mid R} \bigg( \frac{h(p)}{p} \bigg( 1 - \frac{h(p)}{p} \bigg)^{\alpha_p} + \bigg( {-} \frac{h(p)}{p} \bigg)^{\alpha_p} \bigg( 1 - \frac{h(p)}{p} \bigg) \bigg).
\]
Expanding this product results in a sum where each summand is a product over all the prime divisors of $R$. To each summand we assign a squarefree divisor $d \mid R$ by setting $d$ equal to the product of those primes contributing a factor of the form $(h(p)/p)(1 - h(p)/p)^{\alpha_p}$ to the summand; this accounts for all such divisors $d$ of $R$. Therefore,
\begin{align*}
H(r) &= \sum_{d \mid R} \bigg( \prod_{p \mid d} \frac{h(p)}{p} \bigg( 1 - \frac{h(p)}{p} \bigg)^{\alpha_p} \bigg) \bigg( \prod_{p \mid R/d} \bigg( {-} \frac{h(p)}{p} \bigg)^{\alpha_p} \bigg( 1 - \frac{h(p)}{p} \bigg) \bigg) \\
&= \sum_{d \mid R} \frac{h(d)}{d} \bigg( \prod_{p \mid d} \bigg( 1 - \frac{h(p)}{p} \bigg)^{\alpha_p} \prod_{p \mid R/d} \bigg( {-} \frac{h(p)}{p} \bigg)^{\alpha_p} \bigg) \prod_{p \mid R/d}\bigg( 1 - \frac{h(p)}{p} \bigg)\\
&= \sum_{d \mid R} \frac{h(d)}{d} f_r(d) \prod_{p \mid R/d} \bigg( 1 - \frac{h(p)}{p} \bigg)
\end{align*}
as claimed, where the last equality uses the definition~\eqref{f def} of~$f_r(d)$.

Now, suppose that $s$ is a divisor of~$R$ (hence itself squarefree). From Definition~\ref{H and E definition},
\[
J(r,s) = \prod_{p \mid R/s} \bigg( {-}\frac{h(p)}p \bigg)^{\alpha_p} \prod_{p \mid s} \bigg( \bigg( 1-\frac{h(p)}p \bigg)^{\alpha_p} - \bigg( {-}\frac{h(p)}p \bigg)^{\alpha_p} \bigg).
\]
Expanding the second product using the same method as above, we obtain
\begin{align*}
J(r,s) &= \prod_{p \mid R/s} \bigg( {-} \frac{h(p)}{p} \bigg)^{\alpha_p} \sum_{d \mid s} \bigg( \prod_{p \mid d} \bigg( 1 - \frac{h(p)}{p} \bigg)^{\alpha_p} \bigg) \bigg( \prod_{p \mid s/d} - \bigg( {-} \frac{h(p)}{p} \bigg)^{\alpha_p} \bigg) \\
&= \sum_{d \mid s} \bigg( \prod_{p \mid d} \bigg( 1 - \frac{h(p)}{p} \bigg)^{\alpha_p} \bigg) \bigg( \prod_{p \mid R/d} \bigg( {-} \frac{h(p)}{p} \bigg)^{\alpha_p} \bigg) \bigg( \prod_{p \mid s/d} -1 \bigg) \\
&= \sum_{d \mid s} f_r(d) \mu(s/d)
\end{align*}
(using equation~\eqref{f def} again), which is equivalent to the second assertion of the lemma.
\end{proof}

We now present several lemmas which will aid in the proof of Proposition \ref{summing products of k Fs}, which is the main result of this section. Proposition \ref{summing products of k Fs} handles expressions of the form
\[
\sum_{a \in \A} \prod_{j = 1}^k F^\P_{g_j}(a),
\]
providing an asymptotic formula when $k$ is even and an upper bound when $k$ is odd. Lemmas \ref{paired off lemma} through \ref{fra calculation lemma} serve to streamline the proof of Proposition \ref{summing products of k Fs}.

In the following formulas and proofs, all implied constants may depend on the positive integer $k$, which in practice will satisfy $k \leq \delta m$ (recall that $\delta$ is the degree of the polynomial~$Q$), in addition to the finite sets $\A$ and $\P$ and the constant $G$.

\begin{remark}
We note that the lemmas below are nearly identical to those in~\cite[Section 5]{mt18}, with slightly different notation; for example, the expression $\operatorname{cov}(g_1, g_2)$ in \cite{mt18} is replaced by the slightly more general $\kappa(g_1, g_2)$ here, which is defined in nearly the same way, save for the multiplicative factor $h(p)$ measuring the local densities in our sieve setup. Similarly, the function $H(n)$ in this manuscript is nearly the same as $H(n)$ in \cite{mt18}, save for the presence of this same $h(p)$. However, we give the proofs of these formulas here, as the situation in \cite{mt18} (involving one additive function with much larger mean than the others) required somewhat more complicated handling.
\end{remark}

\begin{lemma} \label{paired off lemma}
Let $k$ be a positive even integer and let $g_1, \ldots, g_k$ be any strongly additive functions. Then
\begin{multline} \label{paired off formula}
\sum_{\substack{p_1,\dots,p_k\in\P \\ p_1\cdots p_k \text{ squarefull} \\ \#\{p_1,\dots,p_k\} = k/2}} H(p_1\cdots p_k) g_1(p_1) \cdots g_k(p_k) \\
= \frac 1{(k/2)!} \sum_{\tau\in T_k} \prod_{j=1}^{k/2} \kappa^\P(g_{\Upsilon_1(j)},g_{\Upsilon_2(j)}) + O\big(\mfK^{k/2-1} \big),
\end{multline}
where $T_k$, $\Upsilon_1$, and $\Upsilon_2$ are as in Definition \ref{Tk def} and $\kappa^\P(g_i, g_j)$ is the covariance of $g_i$ and $g_j$ as defined in Section \ref{notation section}.
\end{lemma}

\begin{proof}
To each $k$-tuple $(p_1,\dots,p_k)$ counted by the sum on the left-hand side, we can uniquely associate a $(k/2)$-tuple $(q_1,\dots,q_{k/2})$ of primes satisfying $q_1 < \cdots < q_{k/2}$ and each $q_j$ equals exactly two $p_i$ (at least two because $p_1\cdots p_k$ is squarefull, and then exactly two because $\#\{p_1,\dots,p_k\} = k/2$). This defines a unique $\tau\in T_k$, namely $\tau(i)$ equals the integer $j$ such that $p_i=q_j$. Therefore
\begin{align}
\sum_{\substack{p_1,\dots,p_k\in\P \\ p_1\cdots p_k \text{ squarefull} \\ \#\{p_1,\dots,p_k\} = k/2}} H(p_1\cdots p_k) & g_1(p_1) \cdots g_k(p_k) \notag \\
&= \sum_{\tau\in T_k} \sum_{\substack{q_1 < \cdots < q_{k/2} \\ q_j\in\P}} H(q_1^2 \cdots q_{k/2}^2) g_1(q_{\tau(1)}) \cdots g_k(q_{\tau(k)}) \label{!} \\
&= \frac1{(k/2)!} \sum_{\tau\in T_k} \sum_{\substack{q_1, \dots, q_{k/2} \in\P \\ q_1,\dots,q_{k/2} \text{ distinct}}} H(q_1^2 \cdots q_{k/2}^2) g_1(q_{\tau(1)}) \cdots g_k(q_{\tau(k)}). \notag
\end{align}
By the multiplicativity of $H$ and its values~\eqref{Hp^2} on squares of primes, we see that
\begin{align}
\sum_{\substack{p_1,\dots,p_k\in\P \\ p_1\cdots p_k \text{ squarefull} \\ \#\{p_1,\dots,p_k\} = k/2}} & H(p_1\cdots p_k) g_1(p_1) \cdots g_k(p_k) \notag \\
&= \frac1{(k/2)!} \sum_{\tau\in T_k} \sum_{\substack{q_1, \dots, q_{k/2} \in\P \\ q_1,\dots,q_{k/2} \text{ distinct}}} \bigg( \prod_{j=1}^{k/2} \frac{h(q_j)}{q_j} \bigg( 1-\frac{h(q_j)}{q_j} \bigg) \bigg) g_1(q_{\tau(1)}) \cdots g_k(q_{\tau(k)}) \notag \\
&= \frac1{(k/2)!} \sum_{\tau\in T_k} \sum_{\substack{q_1, \dots, q_{k/2} \in\P \\ q_1,\dots,q_{k/2} \text{ distinct}}} \prod_{j=1}^{k/2} g_{\Upsilon_1(j)}(q_j) g_{\Upsilon_2(j)}(q_j) \frac{h(q_j)}{q_j} \bigg( 1-\frac{h(q_j)}{q_j} \bigg). \label{hi}
\end{align}
If we fix $\tau$ and $q_1,\dots,q_{k/2-1}$, the innermost sum over $q_{k/2}$ is
\begin{multline*}
\sum_{\substack{q_{k/2}\in\P \\ q_{k/2} \notin \{q_1,\dots,q_{k/2-1}\}}} g_{\Upsilon_1(k/2)}(q_{k/2}) g_{\Upsilon_2(k/2)}(q_{k/2}) \frac{h(q_{k/2})}{q_{k/2}} \bigg( 1-\frac{h(q_{k/2})}{q_{k/2}} \bigg) \\
= \kappa^\P(g_{\Upsilon_1(k/2)},g_{\Upsilon_2(k/2)}) + O(1).
\end{multline*}
Summing in turn over $q_{k/2-1},\dots,q_1$ in the same way, equation~\eqref{hi} becomes
\begin{align}
\sum_{\substack{p_1,\dots,p_k\in\P \\ p_1\cdots p_k \text{ squarefull} \\ \#\{p_1,\dots,p_k\} = k/2}} H( & p_1\cdots p_k) g_1(p_1) \cdots g_k(p_k) \notag \\
&= \frac 1{(k/2)!} \sum_{\tau\in T_k} \prod_{j=1}^{k/2} \big( \kappa^\P(g_{\Upsilon_1(j)},g_{\Upsilon_2(j)}) + O(1) \big) \notag \\
&= \frac 1{(k/2)!} \sum_{\tau\in T_k} \bigg( \prod_{j=1}^{k/2} \kappa^\P(g_{\Upsilon_1(j)},g_{\Upsilon_2(j)}) + O( \mfK^{k/2-1}) \bigg) \label{@}
\end{align}
as desired, since each $\kappa^\P(g_{\Upsilon_1(j)},g_{\Upsilon_2(j)}) \ll \mfK$ by Definition~\ref{mf def}.
\end{proof}

\begin{lemma}
\label{unpaired off lemma}
Let $k$ be an integer, and let $g_1, \ldots, g_k$ be strongly additive functions satisfying $0\le g_j(p) \leq G$ for all primes $p$ and all $1 \leq j \leq k$. When $k$ is even,
\begin{equation*}
\sum_{\substack{p_1,\dots,p_k\in\P \\ p_1\cdots p_k \text{ squarefull} \\ \#\{p_1,\dots,p_k\} < k/2}} H(p_1\cdots p_k) g_1(p_1) \cdots g_k(p_k) \ll \mfK^{k/2-1},
\end{equation*}
while when $k$ is odd,
\begin{equation*}
\sum_{\substack{p_1,\dots,p_k\in\P \\ p_1\cdots p_k \text{ squarefull} \\ \#\{p_1,\dots,p_k\} < k/2}} H(p_1\cdots p_k) g_1(p_1) \cdots g_k(p_k) \ll \mfK^{(k-1)/2}.
\end{equation*}
\end{lemma}

\begin{proof}
To each $k$-tuple $(p_1,\dots,p_k)$ counted by the sum, we associate the positive integer $s = \#\{p_1,\dots,p_k\}$, the primes $q_1<\cdots<q_s$ in $\P$ such that $\{q_1,\dots,q_s\} = \{p_1,\dots,p_k\}$, and the integers $\alpha_1,\dots,\alpha_s\ge2$ such that $q_j$ equals exactly $\alpha_j$ of the~$p_i$. Let $T_\alpha$ denote the set of functions from $\{1,\dots,k\}$ to $\{1,\dots,s\}$ such that for each $1\le j\le s$, exactly $\alpha_j$ elements of $\{1,\dots,k\}$ are mapped to~$j$. Define $\Upsilon_1(j)$ and $\Upsilon_2(j)$ to be two of the preimages in $\{1,\dots,k\}$; we will never put ourselves in a situation where we need to know which two. We then have
\begin{multline*}
\sum_{\substack{p_1,\dots,p_k\in\P \\ p_1\cdots p_k \text{ squarefull} \\ \#\{p_1,\dots,p_k\} < k/2}} H(p_1\cdots p_k) g_1(p_1) \cdots g_k(p_k) \\
= \sum_{1\le s<k/2} \sum_{\substack{\alpha_1,\dots,\alpha_s\ge2 \\ \alpha_1+\cdots+\alpha_s=k}} \sum_{\tau\in T_\alpha} \sum_{\substack{q_1 < \cdots < q_{k/2} \\ q_j\in\P}} H(q_1^{\alpha_1}\cdots q_s^{\alpha_s}) g_1(q_{\tau(1)}) \cdots g_k(q_{\tau(k)}).
\end{multline*}
By Lemma~\ref{simple H lemma}(b), $|H(q_1^{\alpha_1} \cdots q_s^{\alpha_s})| \leq H(q_1^2 \cdots q_s^2)$; note also that
\[
\sum_{\tau \in T_\alpha} g_1(q_{\tau(1)}) \cdots g_k(q_{\tau(k)}) \leq G^{k - 2s} \prod_{j = 1}^s g_{\Upsilon_1(j)}(q_j) g_{\Upsilon_2(j)}(q_j).
\]
Therefore
\begin{multline*}
\sum_{\substack{p_1,\dots,p_k\in\P \\ p_1\cdots p_k \text{ squarefull} \\ \#\{p_1,\dots,p_k\} < k/2}} H(p_1\cdots p_k) g_1(p_1) \cdots g_k(p_k) \\
\ll \sum_{1\le s<k/2} \sum_{\substack{\alpha_1,\dots,\alpha_s\ge2 \\ \alpha_1+\cdots+\alpha_s=k}} \sum_{\tau\in T_\alpha} \sum_{\substack{q_1 < \cdots < q_{k/2} \\ q_j\in\P}} H(q_1^2 \cdots q_s^2) g_{\Upsilon_1(j)}(q_j) g_{\Upsilon_2(j)}(q_j).
\end{multline*}
The innermost double sum can be evaluated just as in equations~\eqref{!} and~\eqref{@}, yielding
\begin{align*}
\sum_{\substack{p_1,\dots,p_k\in\P \\ p_1\cdots p_k \text{ squarefull} \\ \#\{p_1,\dots,p_k\} < k/2}} H(p_1\cdots p_k) & g_1(p_1) \cdots g_k(p_k) \\
&\ll \sum_{1\le s<k/2} \sum_{\substack{\alpha_1,\dots,\alpha_s\ge2 \\ \alpha_1+\cdots+\alpha_s=k}} \frac 1{s!} \sum_{\tau\in T_\alpha} \prod_{j=1}^s \big( \kappa^\P(g_{\Upsilon_1(j)},g_{\Upsilon_2(j)}) + O(1) \big) \\
&\ll \sum_{1\le s<k/2} \sum_{\substack{\alpha_1,\dots,\alpha_s\ge2 \\ \alpha_1+\cdots+\alpha_s=k}} \frac 1{s!} \sum_{\tau\in T_\alpha} \mfK^s \ll \mfK^{\max\{s\in\N\colon s<k/2\}}
\end{align*}
as desired, since the implicit constant is allowed to depend upon~$k$.
\end{proof}

\begin{lemma} \label{fra calculation lemma}
For any positive integer $r$, we have
\[
\sum_{a\in\A} f_r(a) = H(r)X + \sum_{s\mid r} \mu^2(s) J(r,s) E_s,
\]
where $E_s$ is defined by equation~\eqref{sieveequation}.
\end{lemma}

\begin{proof}
This is~\cite[equation~(13)]{gs07}; for completeness, we give the full argument here. Let $R$ denote the largest squarefree divisor of $r$, and note from equation~\eqref{f def} that $f_r(a) = f_r(d)$ whenever $(a,R)=d$. Then
\begin{align*}
\sum_{a \in \A} f_r(a) &= \sum_{d \mid R} f_r(d) \sum_{\substack{a \in \A \\ (a, R) = d}} 1 = \sum_{d \mid R} f_r(d) \sum_{\substack{a \in \A \\ d\mid a}} \sum_{e\mid(a/d,R/d)} \mu(e) \\
&= \sum_{d \mid R} f_r(d) \sum_{e\mid R/d} \mu(e) \sum_{\substack{a \in \A \\ de\mid a}} 1 \\
&= \sum_{d \mid R} f_r(d) \sum_{e\mid R/d} \mu(e) \bigg( \frac{h(de)}{de} X + E_{de} \bigg)
\end{align*}
by equation~\eqref{sieveequation}. Since $R$ is squarefree, we have $h(de)=h(d)h(e)$, and therefore
\begin{align*}
\sum_{a \in \A} f_r(a) &= X\sum_{d \mid R} f_r(d) \frac{h(d)}{d} \sum_{e\mid R/d} \frac{\mu(e)h(e)}e + \sum_{d \mid R} f_r(d) \sum_{e \mid R/d} \mu(e)E_{de} \\
&= X\sum_{d \mid R} f_r(d) \frac{h(d)}{d} \prod_{p \mid R/d} \bigg( 1 - \frac{h(p)}{p} \bigg) + \sum_{s\mid R} E_s \sum_{de=s} f_r(d)\mu(e) \\
&= H(r)X + \sum_{s \mid R} E_s J(r, s)
\end{align*}
by Lemma~\ref{bonus lemma}, which is equivalent to the assertion of the lemma.
\end{proof}

\begin{lemma} \label{same as their error term lemma}
Let $k$ be a positive integer, and let $g_1, \ldots, g_k$ be strongly additive functions satisfying $0\le g_j(p) \leq G$ for all primes $p$ and all $1 \leq j \leq k$. Then
\[
\sum_{p_1,\dots,p_k \in\P} g_1(p_1) \cdots g_k(p_k) \sum_{s\mid p_1\cdots p_k} \mu^2(s) J(p_1\cdots p_k,s) E_s \ll \mu^\P(1)^k \sum_{d\in \D_k(\P)} |E_d|,
\]
where $E_s$ and $E_d$ are defined in equation~\eqref{sieveequation}.
\end{lemma}

\begin{proof}
By Lemma \ref{simple H lemma}(c), 
\[
|J(p_1\cdots p_k,s)| \leq \prod_{\substack{1\le i\le k \\ p_i \nmid s}} \frac{h(p_i)}{p_i}.
\]
Since $g_j(p_j) \ll_G 1$ for all $j$,
\begin{align*}
\sum_{p_1,\dots,p_k \in\P} g_1(p_1) \cdots g_k(p_k) \sum_{s\mid p_1\cdots p_k} &J(p_1\cdots p_k,s) E_s  \\ &\ll \sum_{\ell=0}^k \sum_{\substack{s = q_1 \cdots q_\ell \geq 1 \\ q_1 < \cdots < q_\ell \in \P}} |E_s| \sum_{\substack{p_1, \ldots, p_k \in \P \\ s \mid p_1 \cdots p_k}} \prod_{\substack{1\le i\le k \\ p_i \nmid s}} \frac{h(p_i)}{p_i}.
\end{align*}
At the cost of a constant depending only on $k$, we can arrange the primes $p_1, \ldots, p_k$ so that $p_1 = q_1$,  \dots, $p_\ell = q_\ell$. Summing the factor $h(p_i)/p_i$ over the remaining primes $p_i$ for $\ell+1\le i\le k$ yields $\mu^\P(1)^{k - \ell} \ll \mu^\P(1)^k$. Each term $|E_s|$ is equal to $|E_d|$ for some integer $d \in \D_k(\P)$, and therefore the total error term is $\ll \mu^\P(1)^k \sum_{d \in \D_k(\P)} |E_d|$, as desired.
\end{proof}

WIth these preliminary calculations out of the way, we may now establish the main proposition of this section, which will feature prominently in the proof of Theorem~\ref{moment theorem} in the next section.


\begin{prop}
\label{summing products of k Fs}
Let $g_1, \ldots, g_k$ be strongly additive functions satisfying $0 \leq g_j(p) \leq G$ for all primes $p$ and all $1 \leq j \leq k$. Let $F_g^\P$ be as in equation~\eqref{F_g def}, and let $k$ be a positive integer. If $k$ is even, then
\begin{align*}
\sum_{a\in\A} \prod_{j=1}^k F_{g_j}^\P(a) = \frac X{(k/2)!} & \sum_{\tau\in T_k} \prod_{j=1}^{k/2} \kappa^\P(g_{\Upsilon_1(j)},g_{\Upsilon_2(j)}) \\
&+ O\bigg( X \mfK^{k/2-1} + \mu^\P(1)^k \sum_{d\in \D_k(\P)} |E_d| \bigg).
\end{align*}
If $k$ is odd, then
\[
\sum_{a\in\A} \prod_{j=1}^k F_{g_j}^\P(a) \ll X \mfK^{(k-1)/2} + \mu^\P(1)^k \sum_{d\in \D_k(\P)} |E_d|.
\]
\end{prop}

\begin{proof}
Expanding out the definition,
\[
\sum_{a\in\A} \prod_{j=1}^k F_{g_j}^\P(a) = \sum_{a\in\A} \prod_{j=1}^k \sum_{p\in\P} g(p)f_p(a) =  \sum_{p_1,\dots,p_k \in\P} g_1(p_1) \cdots g_k(p_k) \sum_{a\in\A} f_{p_1\cdots p_k}(a).
\]
By Lemma~\ref{fra calculation lemma},
\begin{align*}
\sum_{a\in\A} \prod_{j=1}^k F_{g_j}^\P(a) &= \sum_{p_1,\dots,p_k \in\P} g_1(p_1) \cdots g_k(p_k) \bigg( H(p_1\cdots p_k)X + \sum_{s\mid p_1\cdots p_k} \mu^2(s) J(p_1\cdots p_k,s) E_s \bigg) \notag \\
&= X \sum_{p_1,\dots,p_k \in\P} H(p_1\cdots p_k) g_1(p_1) \cdots g_k(p_k) + O\bigg( \mu^\P(1)^k \sum_{d\in \D_k(\P)} |E_d| \bigg),
\end{align*}
using Lemma~\ref{same as their error term lemma} to obtain the error term. Since $H(p_1\cdots p_k)$ vanishes unless $p_1\cdots p_k$ is squarefull by Lemma~\ref{simple H lemma}(a), there are at most $k/2$ distinct primes among $p_1,\dots,p_k$, and so we can write
\begin{multline*}
\sum_{p_1,\dots,p_k \in\P} H(p_1\cdots p_k) g_1(p_1) \cdots g_k(p_k) = \sum_{\substack{p_1,\dots,p_k\in\P \\ p_1\cdots p_k \text{ squarefull} \\ \#\{p_1,\dots,p_k\} = k/2}} H(p_1\cdots p_k) g_1(p_1) \cdots g_k(p_k) \\
+ \sum_{\substack{p_1,\dots,p_k\in\P \\ p_1\cdots p_k \text{ squarefull} \\ \#\{p_1,\dots,p_k\} < k/2}} H(p_1\cdots p_k) g_1(p_1) \cdots g_k(p_k).
\end{multline*}
The proposition now follows upon appeals to Lemmas~\ref{paired off lemma} and~\ref{unpaired off lemma} (and the observation that the first sum on the right-hand side is empty if $k$ is odd).
\end{proof}

\begin{cor}
\label{upper bound cor}
Under the hypotheses of Theorem~\ref{moment theorem}, for any positive integer $k$, 
\[
\sum_{a\in\A} \prod_{j=1}^k \big| F_{g_j}^\P(a) \big| \ll \begin{cases}
X \mfK^{k/2}, &\text{if $k$ is even}, \\
X \mfK^{(k-1)/2}, &\text{if $k$ is odd}.
\end{cases}
\]
\end{cor}

\begin{proof}
First, note that if the estimate~\eqref{stronger error bound} holds, then it holds with $\delta m$ replaced throughout by any integer $k \leq \delta m$. Indeed, under the hypotheses of Theorem~\ref{moment theorem}, we have $\mu^\P(1) \geq \tfrac{1}{G} \mu^\P(g_j) \asymp \mfK$ for any $1 \leq j \leq k$ (as in the deduction of Theorem~\ref{main theorem sieve version}). Since $\D_k(\P) \subseteq \D_{\delta m}(\P)$, we see that
\begin{align*}
\mu^\P(1)^{k} \sum_{d \in \D_k(\P)} |E_d| &\leq \frac1{\mu^\P(1)^{\delta m - k}} \cdot {\mu^\P(1)^{\delta m}} \sum_{d \in \D_{\delta m}(\P)} |E_d| \\
&\ll \frac1{\mu^\P(1)^{\delta m - k}} \cdot {X \mfK^{\delta m/2 - 1}} \\
&\ll X \mfK^{k - \delta m/2 - 1} \leq X \mfK^{k/2 - 1}.
\end{align*}

If $k$ is even, then by Proposition~\ref{summing products of k Fs} and the triangle inequality,
\begin{align*}
\sum_{a\in\A} \prod_{j=1}^k \big| F_{g_j}^\P(a)\big| &\le \frac X{(k/2)!} \sum_{\tau\in T_k} \prod_{j=1}^{k/2} \mfK + O\bigg( X \mfK^{k/2-1} +  \mu^\P(1)^k \sum_{d\in \D_k(\P)} |E_d| \bigg) \\
&\ll X\mfK^{k/2} + O\bigg( X \mfK^{k/2-1} + \mu^\P(1)^k \sum_{d\in \D_k(\P)} |E_d| \bigg) \ll X \mfK^{k/2}
\end{align*}
by equation~\eqref{stronger error bound}. On the other hand, if $k$ is odd, then again by Proposition~\ref{summing products of k Fs} and equation~\eqref{stronger error bound},
\begin{align*}
\sum_{a\in\A} \prod_{j=1}^k \big| F_{g_j}^\P(a)\big| &\ll X\mfK^{(k-1)/2} + \mu^\P(1)^k \sum_{d\in \D_k(\P)} |E_d| \ll X\mfK^{(k-1)/2},
\end{align*}
as desired.
\end{proof}

\section{Calculating the moments}

We are finally ready to compute the moments, using the technology built up in prior sections. The reader will observe that the computation ends up being rather brief, thanks to this earlier work. By equation~\eqref{Mm with mus and Fs} and Definition~\ref{Rh def},
\begin{align}\label{Mh expanded}
M_m &= \sum_{a\in\A} \big( Q(F_{g_1}^\P(a)+\mu^\P(g_1),\dots,F_{g_\ell}^\P(a)+\mu^\P(g_\ell)) - Q(\mu^\P(g_1),\dots,\mu^\P(g_\ell)) \big)^m \notag \\
&= \sum_{a\in\A} R_m \big( F_{g_1}^\P(a),\dots,F_{g_\ell}^\P(a),\mu^\P(g_1),\dots,\mu^\P(g_\ell) \big) \nonumber \\
&= \sum_{a\in\A} \sum_{\beta=1}^{B_m} r_{m\beta} \prod_{i=1}^{k_{m\beta}} F_{g_{v(m,\beta, i)}}^\P(a) \prod_{j=1}^{\tilde k_{m\beta}} \mu^\P(g_{w(m,\beta, j)}) \nonumber \\
&= \sum_{\beta=1}^{B_m} r_{m\beta} \prod_{j=1}^{\tilde k_{m\beta}} \mu^\P(g_{w(m,\beta, j)}) \sum_{a\in\A} \prod_{i=1}^{k_{m\beta}} F_{g_{v(m,\beta, i)}}^\P(a),
\end{align}
where we used equation~\eqref{Rh expanded} in the middle equality.

We analyze the expression~\eqref{Mh expanded} to obtain Theorem \ref{moment theorem}, splitting our work into two cases depending on whether $m$ is even or odd. We remind the reader that all implied constants may depend on the polynomial $Q$, the constant $G$ from the hypotheses of Theorem \ref{moment theorem}, the positive integer~$m$, and the finite set~$\A$.

\begin{proof}[Proof of Theorem \ref{moment theorem} when $m$ is even]
Our goal is to show that
\[
M_m = C_m X B_Q(\P)^m + O\big( X \mfM^{\delta m-(m+1)/2} \big).
\]
We first isolate the terms in the expression~\eqref{Mh expanded} with $k_{h\beta}=m$, the minimum possible (see the discussion following Definition~\ref{Rh def}), with the goal of showing that these terms contribute the main term of~$M_m$. By Proposition~\ref{summing products of k Fs}, we have for these terms
\begin{align*}
\prod_{i=1}^{k_{m\beta}} & F_{g_{v(m,\beta, i)}}^\P(a) \\
&= \frac X{(m/2)!} \sum_{\tau\in T_m} \prod_{i=1}^{m/2} \kappa^\P(g_{v(m,\beta,\Upsilon_1(i))},g_{v(m,\beta,\Upsilon_2(i))}) + O\bigg( X\mfK^{m/2-1} + \mu^\P(1)^m \sum_{d\in \D_m(\P)} |E_d| \bigg) \\
&= \frac X{(m/2)!} \sum_{\tau\in T_m} \prod_{i=1}^{m/2} \kappa^\P(g_{v(m,\beta,\Upsilon_1(i))},g_{v(m,\beta,\Upsilon_2(i))}) + O\big( X\mfK^{m/2-1} \big)
\end{align*}
since we assume the error bound~\eqref{stronger error bound} (which, as discussed in the proof of Corollary~\ref{upper bound cor}, holds for the smaller integer $m$ since it is assumed to hold for $\delta m$). Therefore
\begin{align}
\sum_{\substack{\beta\le B_m \\ k_{m\beta}=m}} r_{m\beta} & \prod_{j=1}^{\tilde k_{m\beta}} \mu^\P(g_{w(m,\beta, j)}) \sum_{a\in\A} \prod_{i=1}^{k_{m\beta}} F_{g_{v(m,\beta, i)}}^\P(a) \notag \\
&= \sum_{\substack{\beta\le B_m \\ k_{h\beta}=m}} r_{m\beta} \prod_{j=1}^{\tilde k_{m\beta}} \mu^\P(g_{w(m,\beta, j)}) \notag \\
&\qquad{}\times \bigg( \frac X{(m/2)!} \sum_{\tau\in T_m} \prod_{i=1}^{m/2} \kappa^\P(g_{v(m,\beta,\Upsilon_1(i))},g_{v(m,\beta,\Upsilon_2(i))}) + O\big( X\mfK^{m/2-1} \big) \bigg) \notag \\
&= \frac X{(m/2)!} \sum_{\substack{\beta\le B_m \\ k_{m\beta}=m}} r_{m\beta} \prod_{j=1}^{\tilde k_{m\beta}} \mu^\P(g_{w(m,\beta, j)}) \sum_{\tau\in T_m} \prod_{i=1}^{m/2} \kappa^\P(g_{v(m,\beta,\Upsilon_1(i))},g_{v(m,\beta,\Upsilon_2(i))}) \notag \\
&\qquad{}+ O\big( X\mfM^{(\delta-1)m} \mfK^{m/2-1} \big), \label{m main error}
\end{align}
since $\tilde k_{m\beta} \le \delta m-k_{m\beta} = (\delta-1)m$.
By Proposition~\ref{Rh magic Phi lemma}, where $y_j$ has been replaced by $\mu^\P(g_j)$ and $z_{ij}$ by $\kappa^\P(g_i,g_j)$, this main term equals
\begin{multline} \label{m main}
C_mX \bigg( \sum_{i=1}^\ell \sum_{j=1}^\ell Q_i \big( \mu^\P(g_1),\dots,\mu^\P(g_\ell) \big) Q_j\big( \mu^\P(g_1),\dots,\mu^\P(g_\ell) \big) \kappa^\P(g_i,g_j) \bigg)^{m/2} \\
= C_mX \big( B_Q(\P)^2 \big)^{m/2}
\end{multline}
by equation~\eqref{BQP def}.
As for the error term, since we assume $\mfK \leq G \mfM$,
\begin{equation}  \label{m error}
X\mfM^{(\delta-1)m} \mfK^{m/2-1} \ll X \mfM ^{\delta m - m/2 - 1}.
\end{equation}

The remaining terms in the expression~\eqref{Mh expanded} have $m < k_{m\beta} \leq \delta m$. For these terms,
\begin{align*}
\sum_{\substack{\beta\le B_m \\ k_{m\beta}>m}} r_{m\beta} \prod_{j=1}^{\tilde k_{m\beta}} \mu^\P(g_{w(m,\beta, j)}) &\sum_{a\in\A} \prod_{i=1}^{k_{m\beta}} F_{g_{v(m,\beta, i)}}^\P(a) \\
&\ll \sum_{\substack{\beta\le B_m \\ k_{m\beta}>m}} |r_{m\beta}| \prod_{j=1}^{\tilde k_{m\beta}} \mu^\P(g_{w(m,\beta, j)}) \sum_{a\in\A} \prod_{i=1}^{k_{m\beta}} \big| F_{g_{v(m,\beta, i)}}^\P(a) \big| \\
&\ll \sum_{\substack{\beta\le B_m \\ k_{m\beta}>m}} |r_{m\beta}| \prod_{j=1}^{\tilde k_{m\beta}} \mu^\P(g_{w(m,\beta, j)}) X \mfK^{k_{m\beta}/2}
\end{align*}
by Corollary~\ref{upper bound cor} (recalling again that equation ~\eqref{stronger error bound} holds for the integers $k_{m\beta}$ under consideration), giving
\begin{equation*}
\sum_{\substack{\beta\le B_m \\ k_{h\beta}>m}} r_{m\beta} \prod_{j=1}^{\tilde k_{m\beta}} \mu^\P(g_{w(m,\beta, j)}) \sum_{a\in\A} \prod_{i=1}^{k_{m\beta}} F_{g_{v(m,\beta, i)}}^\P(a) 
\ll X \sum_{\substack{\beta\le B_m \\ k_{m\beta}>m}} |r_{m\beta}| \mfM^{\tilde k_{m\beta}} \mfK^{k_{m\beta}/2}.
\end{equation*}
Since we assume $\mfK \leq G \mfM$, this estimate becomes
\begin{equation} \label{> m}
\sum_{\substack{\beta\le B_m \\ k_{m\beta}>m}} r_{m\beta} \prod_{j=1}^{\tilde k_{m\beta}} \mu^\P(g_{w(m,\beta, j)}) \sum_{a\in\A} \prod_{i=1}^{k_{m\beta}} F_{g_{v(m,\beta, i)}}^\P(a) 
\ll X \mfM^{\delta m-(m+1)/2}
\end{equation}
since $\tilde k_{m\beta}+\frac{k_{m\beta}}2 = \tilde k_{m\beta}+k_{m\beta}-\frac{k_{m\beta}}2 \le \tilde k_{m\beta}+k_{m\beta}-\frac{m+1}2 \le \delta m-\frac{m+1}2$ for all terms in the sum.

Gathering the results of the calculations~\eqref{m main error}--\eqref{> m} into the expression~\eqref{Mh expanded} yields
\[
M_m = C_m X B_Q(\P)^m + O\big( X \mfM^{\delta m-(m+1)/2} \big)
\]
as desired.
\end{proof}

A slight modification of the above proof suffices for odd integers~$m$.

\begin{proof}[Proof of Theorem \ref{moment theorem} when $m$ is odd]
Now, our goal is to show that $M_m \ll X \mfM^{\delta m - (m + 1)/2}$. The estimate~\eqref{> m} still holds for those terms in expression~\eqref{Mh expanded} with $k_{m\beta} > m$.
For the remaining terms with $k_{m\beta} = m$, we use Corollary \ref{upper bound cor} to write
\begin{align*}
\sum_{\substack{\beta\le B_m \\ k_{m\beta}=m}} r_{m\beta} & \prod_{j=1}^{\tilde k_{m\beta}} \mu^\P(g_{w(m,\beta, j)}) \sum_{a\in\A} \prod_{i=1}^{k_{m\beta}} F_{g_{v(m,\beta, i)}}^\P(a) \\ &\ll X \mfK^{(m-1)/2}\sum_{\substack{\beta\le B_m \\ k_{m\beta}=m}} |r_{m\beta}| \prod_{j=1}^{\tilde k_{m\beta}} \mu^\P(g_{w(m,\beta, j)}) \\
&\ll X \mfK^{(m-1)/2} \mfM^{\tilde k_{m\beta}} \le X \mfK^{(m-1)/2} \mfM^{(\delta-1)m}
\end{align*}
as before. Since $\mfK \leq G \mfM$, we obtain
\[
\sum_{\substack{\beta\le B_m \\ k_{m\beta}=m}} r_{m\beta}  \prod_{j=1}^{\tilde k_{m\beta}} \mu^\P(g_{w(m,\beta, j)}) \sum_{a\in\A} \prod_{i=1}^{k_{m\beta}} F_{g_{v(m,\beta, i)}}^\P(a) \ll X \mfM^{\delta m - (m + 1)/2}
\]
as desired.
\end{proof}

\section{Examples} \label{examples sec}

We conclude by providing examples of Theorem \ref{main theorem intro version} applied to certain choices of the polynomial $Q$ which may be of independent interest. Example~\ref{powers} (which establishes Corollary~\ref{cor1}) pertains to powers of strongly additive functions; Example~\ref{products} (which establishes Corollary~\ref{cor2}) pertains to products of strongly additive functions; and Example \ref{linear combinations} pertains to linear combinations of strongly additive functions.

\begin{example}\label{powers}
Let $\delta$ be a positive integer and set $Q(T) = T^\delta$, so that $Q'(T) = \delta T ^{\delta - 1}$. Then, for any strongly additive function $g(n)$ satisfying the hypotheses of Theorem \ref{main theorem intro version}, the function $g(n)^\delta$ satisfies an Erd\H os--Kac law with mean $A_Q(x) = \mu(g)^\delta$ and variance 
\begin{align*}
B_Q(x)^2 = \big( Q'(\mu(g)) \big)^2 \kappa(g, g) = \delta^2 \mu(g)^{2\delta - 2} \sigma^2(g).
\end{align*}

For a concrete example, if $g(n) = \omega(n)$, the number-of-distinct-prime-divisors function, then $\omega(n)^\delta$ satisfies an Erd\H os--Kac law with mean $(\log\log n)^\delta$ and variance $\delta^2 (\log\log n)^{2\delta - 2}$; and more generally, the same holds for $Q(g(n))$ for any polynomial $Q(T)$ with nonnegative coefficients and leading term $T^\delta$.
\end{example}

\begin{example}\label{products}
Let $g_1(n)$ and $g_2(n)$ be strongly additive functions satisfying the hypotheses of Theorem \ref{main theorem intro version}, and set $Q(T) = T_1T_2$, so that $\frac{\partial Q}{\partial T_1} = T_2$ and $\frac{\partial Q}{\partial T_2} = T_1$. Then the product $g_1(n) g_2(n)$ satisfies an Erd\H os--Kac law with mean $A_Q(x) = \mu(g_1) \mu(g_2)$ and variance 
\begin{align*}
B_Q(x)^2 &= \sum_{i = 1}^2 \sum_{j = 1}^2 \frac{\partial Q}{\partial T_i}\big(\mu(g_1), \mu(g_2)\big) \frac{\partial Q}{\partial T_j}\big(\mu(g_1), \mu(g_2)\big) \kappa(g_i, g_j) \\
&= \mu(g_1)^2\sigma^2(g_2) + 2\mu(g_1)\mu(g_2)\kappa(g_1,g_2) + \mu(g_2)^2\sigma^2(g_1).
\end{align*}

For example, suppose $P_1$ and $P_2$ are sets of primes of positive relative densities $\alpha$ and $\beta$, such that $P_1\cap P_2$ has relative density $\gamma$ in the primes; and set $\omega_j(n) = \# \{ p\mid n \colon p\in P_j \}$ for $j=1,2$. Using the definitions~\eqref{main theorem mu kappa def}, it is easy to calculate that
\[
\mu(g_1) \sim \sigma^2(g_1) \sim \alpha\log\log x, \quad\mu(g_2) \sim \sigma^2(g_2) \sim \beta\log\log x, \quad \kappa(g_1,g_2) \sim \gamma\log\log x.
\]
Therefore $\omega_1(n)\omega_2(n)$ satisfies an Erd\H os--Kac law with mean $\alpha\beta(\log\log x)^2$ and variance $\alpha\beta(\alpha+2\gamma+\beta)(\log \log x)^3$. In particular, note that the variance depends not just on the two additive functions $g_1$ and $g_2$ individually but also upon the correlations in their values.

More generally, let $g_1(n), \dots, g_\ell(n)$ be strongly additive functions satisfying the hypotheses of Theorem~\ref{main theorem intro version}, and set $Q(T) = T_1\cdots T_\ell$, so that $\frac{\partial Q}{\partial T_j} = (T_1\cdots T_\ell)/T_j$ for all $1\le j\le\ell$. Then the function $g_1(n) \cdots g_\ell(n)$ satisfies an Erd\H os--Kac law with mean $A_Q(x) = \mu(g_1)\cdots \mu(g_\ell)$ and variance
\begin{align*}
B_Q(x)^2 &= \sum_{i = 1}^\ell \sum_{j = 1}^\ell \frac{\partial Q}{\partial T_i}\big(\mu(g_1), \dots, \mu(g_\ell)\big) \frac{\partial Q}{\partial T_j}\big(\mu(g_1), \dots, \mu(g_\ell)\big) \kappa(g_i, g_j) \\
&= \big( \mu(g_1)\cdots \mu(g_\ell) \big)^2 \sum_{1\le i,j\le \ell} \frac{\kappa(g_i,g_j)}{\mu(g_i) \mu(g_j)}.
\end{align*}
\end{example}

\begin{example}\label{linear combinations}
Again let $g_1(n)$ and $g_2(n)$ be strongly additive functions satisfying the hypotheses of Theorem \ref{main theorem intro version}; this time, set $Q(T_1, T_2) = vT_1^\delta + wT_2^\delta$ for a positive integer $\delta$ and nonnegative real numbers $v$ and $w$, so that $\frac{\partial Q}{\partial T_1} = \delta v T_1^{\delta - 1}$ and $\frac{\partial Q}{\partial T_2} = \delta w T_2^{\delta - 1}$. Then the function $vg_1(n)^\delta + wg_2(n)^\delta$ satisfies an Erd\H os--Kac law with mean $v\mu(g_1)^\delta + w\mu(g_2)^\delta$ and variance
\begin{align*}
B_Q(x)^2 &= \sum_{i = 1}^2 \sum_{j = 1}^2 \frac{\partial Q}{\partial T_i}\big(\mu(g_1), \mu(g_2)\big) \frac{\partial Q}{\partial T_j}\big(\mu(g_1), \mu(g_2)\big) \kappa(g_i, g_j) \\
&= v^2 \delta^2 \mu(g_1)^{2\delta - 2} \sigma^2(g_1) + 2vw \delta^2 \big( \mu(g_1) \mu(g_2) \big)^{\delta - 1} \kappa(g_1, g_2) + w^2 \delta^2 \mu(g_2)^{2 \delta - 2}\sigma^2(g_2).
\end{align*}

In particular, when $\delta=1$, the function $g(n) = vg_1(n)+wg_2(n)$ satisfies an Erd\H os--Kac law with mean $v\mu(g_1)+w\mu(g_2)$ and variance $v^2\sigma^2(g_1)+2vw\kappa(g_1,g_2)+w^2\sigma^2(g_2)$. But notice that $g$ itself is an additive function. Using the definitions~\eqref{main theorem mu kappa def}, it is trivial to check that $\mu(g) = v\mu(g_1)+w\mu(g_2)$, and still easy to check that $\sigma^2(g) = v^2\sigma^2(g_1)+2vw\kappa(g_1,g_2)+w^2\sigma^2(g_2)$. Therefore our theorem is self-consistent in this case.
\end{example}

\bibliographystyle{amsplain}
\bibliography{refs}

\end{document}